\documentclass[11pt,a4paper]{amsart}
\usepackage{amssymb,amsmath,amsfonts}
\usepackage{mathtools}
\headsep=1.2500cm \numberwithin{equation}{section}
\newtheorem{Theorem}{Theorem}[section]
\newtheorem{Proposition}[Theorem]{Proposition}
\newtheorem{cor}[Theorem]{Corollary}
\newtheorem{lemma}[Theorem]{Lemma}
\theoremstyle{remark}

\newtheorem{Example}[Theorem]{Example}

\begin{document}
\title{On left $\phi$-biprojectivity and left $\phi$-biflatness of certain Banach algebras}

\author[A. Sahami]{A. Sahami}
\address{Department of Mathematics Faculty of Basic Science, Ilam University, P.O. Box 69315-516 Ilam, Iran.}
\email{a.sahami@ilam.ac.ir}



\begin{abstract}
In this paper, we study left $\phi$-biflatness and left $\phi$-biprojectivity of some Banach algebras, where $\phi$ is a non-zero multiplicative linear function. We show that if the Banach algebra $A^{**}$ is left $\phi$-biprojective, then $A$ is left $\phi$-biflat. Using this tool we study left $\phi$-biflatness of some matrix algebras. We also study left $\phi$-biflatness and left $\phi$-biprojectivity of the projective tensor product of some Banach algebras. We prove that  for a locally compact group $G$, $M(G)\otimes_{p} A(G)$ is left $\phi\otimes \psi$-biprojective if and only if $G$ is finite. We show that  $M(G)\otimes_{p} L^1(G)$ is left $\phi\otimes \psi$-biprojective if and only if $G$ is compact.
\end{abstract}
\maketitle
\section{Introduction and Preliminaries}
Banach homology theory have two
important notions, biflatness and biprojectivity  which have  key role in studying the structure of Banach algebras. A Banach
algebra $A$ is called biflat (biprojective), if there exists a
bounded $A$-bimodule morphism $\rho:A\rightarrow
(A\otimes_{p}A)^{**}$ ($\rho:A\rightarrow A\otimes_{p}A$) such that
$\pi_{A}^{**}\circ\rho$ is the canonical embedding of $A$ into
$A^{**}$ ($\rho$ is a right inverse for $\pi_{A}$), respectively.  
It is well known that for a locally compact group $G$, the group algebra  $L^{1}(G)$ is biflat
(biprojective) if and only if $G$ is amenable (compact),
respectively.
We have to mention that a biflat Banach algebra $A$ with a buonded approximate identity is amenable and vise versa, see \cite{run}.

A  Banach algebra $A$ is called left
$\phi$-amenable, if there exists a bounded net $(a_{\alpha})$ in $A$
such that $aa_{\alpha}-\phi(a)a_{\alpha}\rightarrow 0$ and
$\phi(a_{\alpha})\rightarrow 1$ for all $a\in A,$ where $\phi
\in\Delta(A).$ For a locally compact group $G$, the Fourier algebra
$A(G)$ is always left $\phi$-amenable. Also the group algebra
$L^{1}(G)$ is left $\phi$-amenable if and only if $G$ is amenable,
for further information see \cite{Sang} and
\cite{Ala}.

Following this course, Essmaili et. al. in \cite{rost ch} introduced and studied a biflat-like property related to a multiplicative linear functional, they called it condition $W$ (which we call it here right $\phi$-biflatness). 
The Banach algebra $A$ is called left $\phi$-biflat, if there exists a bounded linear map $\rho:A\rightarrow (A\otimes_{p}A)^{**}$
	such that $$\rho(ab)=\phi(b)\rho(a)=a\cdot \rho(b)$$
	and $$\tilde{\phi}\circ\pi^{**}_{A}\circ\rho(a)=\phi(a),$$
	for each $a,b\in A$.	
We  followed their work and showed that the Segal algebra  $S(G)$ is left  $\phi$-biflat if and only if $G$ is amenable see \cite{sahami left biflat}. also we defined a notion of left $\phi$-biprojectivity for Banach algebras. In fact $A$ Banach algebra is left $\phi$-biprojective if there exists a  bounded linear map $\rho:A\rightarrow A\otimes_{p}A$ such that $$\rho(ab)=a\cdot \rho(b)=\phi(b)\rho(a),\quad \phi\circ\pi_{A}\circ\rho(a)=\phi(a),\quad(a,b\in A).$$
We showed that the Lebesgue-Fourier algebra $LA(G)$ is left $\phi$-contractible if and only if $G$ is compact. Also the Fourier algebra $A(G)$ is left $\phi$-contractible if and only if $G$ is discrete, see \cite{sahami left biprojective}.

In this paper, We show that if the Banach algebra $A^{**}$ is left $\phi$-biprojective, then $A$ is left $\phi$-biflat. Using this tool we study left $\phi$-biflatness of some matrix algebras. We also study left $\phi$-biflatness and left $\phi$-biprojectivity of the projective tensor product of some Banach algebras. We prove that  for a locally compact group $G$, $M(G)\otimes_{p} A(G)$ is left $\phi\otimes \psi$-biprojective if and only if $G$ is finite. We show that  $M(G)\otimes_{p} L^1(G)$ is left $\phi\otimes \psi$-biprojective if and only if $G$ is compact.

We remark some standard notations and definitions that we shall need
in this paper. Let $A$ be a Banach algebra. If $X$ is a Banach
$A$-bimodule, then  $X^{*}$ is also a Banach $A$-bimodule via the
following actions
$$(a\cdot f)(x)=f(x\cdot a) ,\hspace{.25cm}(f\cdot a)(x)=f(a\cdot x ) \hspace{.5cm}(a\in A,x\in X,f\in X^{*}). $$

Throughout, the
character space of $A$ is denoted by $\Delta(A)$,  that is, all
non-zero multiplicative linear functionals on $A$. Let $\phi\in
\Delta(A)$. Then $\phi$ has a unique extension   $\tilde{\phi}\in\Delta(A^{**})$
which is defined by $\tilde{\phi}(F)=F(\phi)$ for every
$F\in A^{**}$.

Let $A$ be a  Banach algebra. The projective tensor product
$A\otimes_{p}A$ is a Banach $A$-bimodule via the following actions
$$a\cdot(b\otimes c)=ab\otimes c,~~~(b\otimes c)\cdot a=b\otimes
ca\hspace{.5cm}(a, b, c\in A).$$
For Banach algebras $A$ and $B$ with $\phi\in\Delta(A)$ and $\psi\in\Delta(B)$, we denote $\phi\otimes\psi$ for a multiplicative linear functional on $A\otimes_{p}B$ given by $\phi\otimes\psi(a\otimes b)=\phi(a)\psi(b)$ for each $a\in A$ and $b\in B.$
The product morphism $\pi_{A}:A\otimes_{p}A\rightarrow A$ is given by $\pi_{A}(a\otimes b)=ab,$ for every $a,b\in A.$
Let $X$ and $Y$ be Banach $A$-bimodules. The map $T:X\rightarrow Y$ is called $A$-bimodule morphism, if
$$T(a\cdot x)=a\cdot T(x),\quad T(x\cdot a)=T(x)\cdot a,\qquad (a\in A,x\in X).$$

\section{Some general properties}
Let $A$ be a Banach algebra and $\phi\in\Delta(A)$. $A$ is called 
	approximate left $\phi$-biprojective if there exists a net of
bounded linear maps from $A$ into $A\otimes_{p}A$, say
$(\rho_{\alpha})_{\alpha\in I}$, such that
\begin{enumerate}
	\item [(i)] $a\cdot \rho_{\alpha}(b)-\rho_{\alpha}(ab)\xrightarrow{||\cdot||} 0$,
	\item [(ii)] $\rho_{\alpha}(ba)-\phi(a)\rho_{\alpha}(b)\xrightarrow{||\cdot||} 0$,
	\item [(iii)] $\phi\circ\pi_{A}\circ\rho_{\alpha}(a)-\phi(a)\rightarrow
	0$,
\end{enumerate}
for every $a,b\in A$, see \cite{sahami approximate left}.
\begin{Proposition}\label{approximate left bi}
	Let $A$ be a left $\phi$-biflat Banach algebra. Then $A$ is approximate left $\phi$-biprojective.
\end{Proposition}
\begin{proof}
Since $A$ is left $\phi$-biflat, there exists a bounded linear map $\rho:A\rightarrow (A\otimes_{p}A)^{**}$ such that $\rho(ab)=a\cdot \rho(b)=\phi(b)\rho(a)$ and  $\tilde{\phi}\circ\pi^{**}_{A}\circ\rho(a)=\phi(a)$. Since $\rho\in B(A,(A\otimes_{p}A)^{**})$, there exists a net $\rho_{\alpha}\in B(A,A\otimes_{p}A)$ such that $\rho_{\alpha}\xrightarrow{W^*OT}\rho$. Thus for each $a\in A$ we have $\rho_{\alpha}(a)\xrightarrow{w^*}\rho(a)$. Then 
$$a\cdot \rho_{\alpha}(b)\xrightarrow{w^*}a\cdot \rho(b)=\rho(ab),\quad \rho_{\alpha}(ab)\xrightarrow{w^*}\rho(ab),\quad \phi(b)\rho_{\alpha}(a)\xrightarrow{w^*}\phi(b)\rho(a)=\rho(ab).$$ 
On the other hand, the map $\pi^{**}_{A}$ is a $w^*$-continuous map, so $\pi^{**}_{A}\circ\rho_{\alpha}(a)\xrightarrow{w^*}\pi^{**}_{A}\circ\rho(a)$, for each $a\in A.$  Then $$\phi\circ\pi_{A}\circ\rho_{\alpha}(a)=\tilde{\phi}\circ \pi^{**}_{A}\circ\rho_{\alpha}(a)=\pi^{**}_{A}\circ\rho_{\alpha}(a)(\phi)\rightarrow \pi^{**}_{A}\circ\rho(a)(\phi)=\tilde{\phi}\circ\pi^{**}_{A}\circ\rho(a)=\phi(a).$$ Also foe each $a,b\in A$, we have 
$$a\cdot \rho_{\alpha}(b)\xrightarrow{w^*}a\cdot \rho(b)=\rho(ab),\quad \rho_{\alpha}(ab)\xrightarrow{w^*}\rho(ab),\quad \phi(b)\rho_{\alpha}(a)\xrightarrow{w^*}\phi(b)\rho(a).$$ So $$a\cdot \rho_{\alpha}(b)-\rho_{\alpha}(ab)\xrightarrow{w^*}0,\quad \phi(b)\rho_{\alpha}(a)-\phi(b)\rho_{\alpha}(a)\xrightarrow{w^*}0.$$
Put $F=\{a_{1},a_{2},...,a_{n}\}$ and $G=\{b_{1},b_{2},...,b_{n}\}$ for finite subsets of $A.$ Define
$$M=\{(a_{1}\cdot T(b_{1})-T(a_{1}b_{1}),a_{2}\cdot T(b_{2})-T(a_{2}b_{2}),...,a_{n}\cdot T(b_{n})-T(a_{n}b_{n})):T\in B(A,A\otimes_{p}A)\},$$
it is easy to see that $M$ is a convex subset of $\prod_{i=1}^{n}(A\otimes_{p}A)\oplus_{1}\prod_{i=1}^{n}\mathbb{C}$ and $(0,0,...,0)\in\overline{M}^{w}=\overline{M}^{||\cdot||}.$ It follows that, there exists a net $\xi_{(\epsilon,F,G)}\in B(A,A\otimes_{p}A)$ such that $$||a_{i}\cdot \xi_{(\epsilon,F,G)}(b_{i})-\xi_{(\epsilon,F,G)}(a_{i}b_{i})||<\epsilon,\quad ||\xi_{(\epsilon,F,G)}(a_{i}b_{i})-\phi(b_{i})\xi_{(\epsilon,F,G)}(a_{i})||<\epsilon$$
and $|\phi\circ\pi_{A}\circ\xi_{(\epsilon,F,G)}(a_{i})-\phi(a_{i})|<\epsilon,$
for each $i\in\{1,2,...,n\}$. It follow that the net $(\xi_{(\epsilon,F,G)})$, for each $a,b\in A$, satisfies
$$a\cdot \xi_{(\epsilon,F,G)}-\xi_{(\epsilon,F,G)}(ab)\rightarrow 0,\quad \phi(b) \xi_{(\epsilon,F,G)}(a)-\xi_{(\epsilon,F,G)}(ab)\rightarrow 0$$
and $$\phi\circ\pi_{A}\circ\xi_{(\epsilon,F,G)}(a)-\phi(a)\rightarrow 0.$$ Therefore $A$ is approximately left $\phi-$biprojective. 
\end{proof}
\begin{lemma}
If $A$ is an approximately left $\phi$-biprojective with bounded net $\rho_{\alpha},$ then $A$ is 
left $\phi$-biflat.
\end{lemma}
\begin{proof}
	Let $A$ be approximately left $\phi$-biprojective with bounded net $\rho_{\alpha}.$ So $\rho_{\alpha}\in B(A,(A\otimes_{p}A)^{**})\cong(A\otimes_{p}(A\otimes_{p}A)^{*})^{*}$ has a $w^*$-limit-point, say $\rho.$ Since  $$a\cdot\rho_{\alpha}(b)-\rho_{\alpha}(ab)\rightarrow 0,\quad \phi(b)\rho_{\alpha}(a)-\rho(ab)\rightarrow 0,\quad \phi\circ\pi_{A}\circ\rho_{\alpha}(a)-\phi(a)\rightarrow 0.$$ It follows that $$a\cdot\rho(b)=\rho(ab)=\phi(b)\rho(a),\quad \tilde{\phi}\circ\pi^{**}_{A}\rho(a)=\phi(a),$$
	for each $a\in A.$
\end{proof}
\begin{Proposition}
	Let $A$ be a Banach algebra with an approximate identity and let $\phi\in\Delta(A)$. If $A^{**}$ is approximately biflat, then $A$ is left $\phi$-biflat.
\end{Proposition}
\begin{proof}
	Since $A$ has an approximate identity $\overline{A\ker\phi}=\ker\phi.$ Thus by \cite[Theorem 3.3]{sahami razi biflat} $A$ is left $\phi-$amenable. So there exists an element $m\in A^{**}$ such that $am=\phi(a)m$ and $\tilde{\phi}(m)=1$ for every $a\in A.$ Define $\rho:A\rightarrow A^{**}\otimes_{p}A^{**}$ by $\rho(a)=\phi(a)m\otimes m$. Claerly $\rho$ is a bounded linear map  such that $$a\cdot\rho(b)=\rho(ab)=\phi(b)\rho(a),\quad \tilde{\phi}\circ\pi_{A^{**}}\circ\rho(a)=\phi(a),\quad (a\in A).$$
	There exists a
	bounded linear map $\psi:A^{**}\otimes_{p} A^{**}\rightarrow
	(A\otimes_{p} A)^{**}$ such that for $a,b\in A$ and $m\in
	A^{**}\otimes_{p} A^{**}$, the following holds;
	\begin{enumerate}
		\item [(i)] $\psi(a\otimes b)=a\otimes b $,
		\item [(ii)] $\psi(m)\cdot a=\psi(m\cdot a)$,\qquad
		$a\cdot\psi(m)=\psi(a\cdot m),$
		\item [(iii)] $\pi_{A}^{**}(\psi(m))=\pi_{A^{**}}(m),$
	\end{enumerate}
	see \cite[Lemma 1.7]{gha}.  Set $\eta=\psi\circ \rho:A\rightarrow (A\otimes_{p}A)^{**}$. It is easy to see that 
$a\cdot\eta(b)=\eta(ab)=\phi(b)\eta(a)$
$$\tilde{\phi}\circ\pi^{**}_{A}\circ\eta(a)=\tilde{\phi}\circ  \pi_{A^{**}}\circ\psi\circ\rho(a)=\tilde{\phi}\circ\pi_{A^{**}}\circ\rho(a)=\phi(a),\quad (a\in A).$$
So $A$ is left $\phi$-biflat.
\end{proof}
Let $A$ be a Banach algebra and $I$ be a totally ordered set. By
$UP_{I}(A)$ we denote  the set of $I\times I$ upper triangular
matrices  which its entries come from $A$ and
$$||(a_{i,j})_{i,j\in I}||=\sum_{i,j\in I}||a_{i,j}||<\infty.$$ With
matrix operations and $||\cdot||$ as a norm, $UP_{I}(A)$ becomes a
Banach algebra.
\begin{Proposition}
Let $I$ be a totally ordered set with the greatest  element. Also let $A$ be a Banach algebra with left identity and      $\phi\in\Delta(A).$ Then $UP(I,A)^{**}$ is left $\psi_{\phi}$-biflat if and only if $|I|=1$ and $A$ is left $\phi$-biflat.
\end{Proposition}
\begin{proof}
Suppose  $UP_{I}(A)$ is left $\psi_{\phi}$-biflat. Let $i_{0}\in I$ be the greatest
element of $I$ with respect to $\leq$. Since $A$ has a left unit, $UP_{I}(A)$ has a left approximate identity. By \cite[Lemma 2.1]{sahami left biflat} left $\psi_{\phi}$-amenability of $UP_{I}(A)^{**}$ implies that  $UP_{I}(A)$ is left $\psi_{\phi}$-amenable.

Define
$$J=\{(a_{i,j})_{i,j\in I}\in UP_{I}(A)|a_{i,j}=0 \quad
\text{for}\quad j\neq i_{0}\}.$$ Clearly  $J$ is a
closed ideal of $UP_{I}(A)$ with $\psi_{i_{0}}|_{J}\neq 0$.
Applying  \cite[Lemma 3.1]{kan} gives that $J$ is left
$\psi_{i_{0}}$-amenable. So by \cite[Theorem
1.4]{kan} there exists a bounded net $(j_{\alpha})$ in $J$ which satisfies
\begin{equation}\label{left phi}
jj_{\alpha}-\psi_{\phi}(j)j_{\alpha}\rightarrow 0,\quad
\psi_{\phi}(j_{\alpha})=1\quad (j\in J).
\end{equation}
 Suppose in contradiction that $I$ has at least two elements.
 Let $a_{0}$ be an element
in $A$ such that $\phi(a_{0})=1.$
Set
$j=\left(\begin{array}{ccccc} \cdots&0&\cdots&0&a_{0}\\
\cdots&0&\cdots&0&a_{0}\\
\colon&\colon&\colon&\colon&\colon\\
\cdots&0&\cdots&0&a_{0}\\
\colon&\colon&\colon&\colon&0
\end{array}
\right).$
Clearly for each $\alpha$
the net
$j_{\alpha}$ has a  form $\left(\begin{array}{ccccc} \cdots&0&\cdots&0&j_{i}^\alpha\\
\cdots&0&\cdots&0&\cdots\\
\colon&\colon&\colon&\colon&\colon\\
\cdots&0&\cdots&0&j_{k}^\alpha\\
\colon&\colon&\colon&\colon&j_{i_{0}}^\alpha
\end{array}
\right)$,
 where $(j_{i}^\alpha),(j_{k}^\alpha)$ and $(j_{i_{0}}^\alpha)$ are some nets in $A$. Put $j$
and $j_{\alpha} $ in (\ref{left phi}) we have
$j_{i_{0}}^{\alpha}a_{0}\rightarrow 0$. Since $\phi$ is continuous, we
have $\phi(j_{i_0}^{\alpha})\rightarrow 0.$ On the other hand
$\psi_{\phi}(j_{\alpha})=\phi(j_{i_0}^{\alpha})=1$ which is a
contradiction. So $I$ must be singleton and the proof is complete.
\end{proof}
\begin{cor}
	Let $I$ be a totally ordered set with the greatest  element. Also let $A$ be a Banach algebra with left identity and      $\phi\in\Delta(A).$ Then $UP_{I}(A)^{**}$ is approximately biflat if and only if $|I|=1$ and $A$ is approximately biflat.
\end{cor}
\begin{Example}
We give a Banach algebra which is not left $\phi$-biflat 
  but it is approximate left $\phi$-biprojective. So the converse of Proposition \ref{approximate left bi} is not always true.
 Let denote 
 $\ell^{1} $  for  the set of all sequences $a=((a_{n}))$ of
complex numbers equipped with $||a||=\sum^{\infty}_{n=1}|a_{n}|<\infty $ as its norm.
With the following product:
\begin{eqnarray*}�
	&\textit{(�$a\ast
		b)(n)=$}\begin{cases}a(1)b(1)\,\,\,\,\,\,\,\,\,\,\,\,\,\,\,\,\,\,\,\,\,\,\,\,\hspace{3.5cm}{\hbox {if}}\qquad
		n=1\cr �
		a(1)b(n)+b(1)a(n)+a(n)b(n)\,\,\,\,\,\,\,\,\,\,\,\,\,\,\,\,\,\,\,\,{\hbox {if}}\qquad
		n>1, �\end{cases}\\�
\end{eqnarray*}
$A=(\ell^{1},||\cdot||)$ becomes a Banach algebra. Clearly  $\Delta(\ell^{1})=\{\phi_{1},\phi_{1}+\phi_{n}\}$, where
$\phi_{n}(a)=a(n)$ for every $a\in \ell^{1}$. We claim that $\ell^{1}$ is not left $\phi_{1}$-biflat but $\ell^{1}$ is approximately left $\tilde{\phi}_{1}$-biprojective for some $\phi\in\Delta(\ell^1)$. We assume conversely that 
$\ell^{1}$ is  left $\phi_{1}$-biflat. One can see that $(1,0,0,...)$ is a unit for $\ell^{1}$. Therefore by \cite[Lemma 2.1]{sahami left biflat} left $\phi_{1}$-biflatness of  $\ell^{1}$ implies
that ${\ell^{1}}$ is left $\phi_{1}$-amenable. On the other hand 
by \cite[Example
2.9]{nem col} $\ell^{1}$ is not left $\phi_{1}$-amenable which is a contradiction.

Applying  \cite[Example 2.9]{nem col}, gives that  $\ell^{1}$ is approximate left
$\phi_{1}$-amenable. So \cite[Proposition 2.4]{sahami approximate left} follows that 
that $\ell^{1}$ is approximate left $\phi_{1}$-biprojective.
\end{Example}
\section{Left $\phi$-biprojectivity of the projective tensor product Banach algebras}
\begin{Theorem}
	Let $A$ and $B$ be Banach algebras which $\phi\in\Delta(A)$ and $\psi\in\Delta(B)$. Suppose that $A$ has a unit and $B$ has an identity $x_{0}$ such that $\psi(x_{0})=1.$ If $A\otimes_{p}B$ is left $\phi\otimes\psi$-biflat, then $A$ is left $\phi$-amenable.
\end{Theorem}
\begin{proof}
	Let $\rho:A\otimes_{p}B\rightarrow ((A\otimes_{p}B)\otimes_{p}(A\otimes_{p}B))^{**}$ be a bounded linear map such that $$\rho(xy)=x\cdot \rho(y)=\widetilde{\phi\otimes \psi}(y)\rho(x),\quad \widetilde{\phi\otimes \psi}\circ\pi^{**}_{A\otimes_{p}B}\circ\rho(x)=\phi\otimes \psi(x)\quad (x,y\in A\otimes_{p}B).$$
	For idempotent $x_{0}\in B$ and elements  $a_{1},a_{2}\in A$  we have 
		$$a_{1}a_{2}\otimes x_{0}=a_{1}a_{2}\otimes x_{0}=a_{1}a_{2}\otimes x_{0}^{2}=(a_{1}\otimes x_{0})(a_{2}\otimes x_{0}).$$ We denote $e$ for the unit of $A$. So we have 
		  \begin{equation*}
		\begin{split}
		\rho(a_{1}a_{2}\otimes x_{0})=\rho((a_{1}\otimes x_{0})(a_{2}\otimes x_{0}))&=(a_{1}\otimes x_{0})\cdot \rho(a_{2}\otimes x_{0})\\
		&=a_{1}(e\otimes x_{0})\cdot \rho(a_{2}\otimes x_{0})\\
		&=a_{1}\rho(ea_{2}\otimes x_{0}^2),
		\end{split}
		\end{equation*}
		also
		$$\rho(a_{1}a_{2}\otimes x_{0})=\rho((a_{1}\otimes x_{0})(a_{2}\otimes x_{0}))=\phi\otimes\psi(a_{2}\otimes x_{0})\rho(a_{1}\otimes x_{0})=\phi(a_{2})\rho(a_{1}\otimes x_{0})$$
		and 
		$$\widetilde{\phi\otimes \psi}\circ\pi^{**}_{A\otimes_{p}B}\circ\rho(a_{1}\otimes x_{0})=\phi\otimes \psi(a_{1}\otimes x_{0})=\phi(a_{1}),$$
		for each $a_{1},a_{2}\in A.$ Put $\xi:(A\otimes_{p}B)\otimes_{p}(A\otimes_{p}B)\rightarrow A\otimes_{p}A$
		for a bounded linear map which is given by $\xi((a\otimes b)\otimes(c\otimes d)=\psi(bd)a\otimes c,$ for each $a,c\in A$
		and $b,d\in B.$ Clearly $$\pi^{**}_{A}\circ \xi^{**}=(id_{A}\otimes\psi)^{**}\circ \pi^{**}_{A\otimes_{p}B}.$$
		Define  $\theta:A\rightarrow (A\otimes_{p}A)^{**}$ by $\theta(a)=\xi^{**}\circ\rho(a\otimes x_{0})$.
		Clearly $\theta $ is a bounded linear map. We have  $$a\cdot \theta(b)=a\cdot \xi^{**}\circ\rho(b)=\xi^{**}\circ\rho(ab)=\phi(b)\xi^{**}\circ\rho(a)=\phi(b)\theta(a),\quad (a,b\in A).$$
		Also 
		 \begin{equation*}
		\begin{split}
	\tilde{\phi}\circ \pi^{**}_{A}\circ\theta(a)=\tilde{\phi}\circ \pi^{**}_{A}\circ\xi^{**}\circ\rho(a\otimes x_{0})&=\tilde{\phi}\circ (id_{A}\otimes\psi)^{**}\circ \pi^{**}_{A\otimes_{p}B}\circ\rho(a\otimes x_{0})\\
	&=\widetilde{\phi\otimes \psi}\circ \pi^{**}_{A\otimes_{p}B}\circ\rho(a\otimes x_{0})\\
	&=\phi\otimes \psi(a\otimes x_{0})=\phi(a),
		\end{split}
		\end{equation*}
for each $a\in A$. It follows that $A$ is left $\phi$-biflat. Since $A$ has a unit by ... $A$ is left $\phi$-amenable.		
\end{proof}
Note that previous theorem is also valid in the left $\phi$-biprojective case. In fact we have 
\begin{cor}
	Let $A$ and $B$ be Banach algebras which $\phi\in\Delta(A)$ and $\psi\in\Delta(B)$. Suppose that $A$ has a unit and $B$ has an identity $x_{0}$ such that $\psi(x_{0})=1.$ If $A\otimes_{p}B$ is left $\phi\otimes\psi$-biprojective, then $A$ is left $\phi$-contractible.
\end{cor}
\begin{Proposition}
Suppose that $A$ is a Banach algebra and $\phi\in\Delta(A).$	Let $A^{**}$ be left $\tilde{\phi}$-biprojective. Then $A$ is left $\phi$-biflat.
\end{Proposition}
\begin{proof}
	Let $A^{**}$ be  $\tilde{\phi}$-biprojective. Then there exists a bounded linear map $\rho:A^{**}\rightarrow A^{**}\otimes_{p}A^{**}$ such that $\rho(ab)=a\cdot \rho(b)=\tilde{\phi}(b)\rho(a)$ and $\tilde{\phi}\circ \pi_{A^{**}}\circ\rho(a)=\tilde{\phi}(a)$, for each $a,b\in A^{**}$. There exists a
	bounded linear map $\psi:A^{**}\otimes_{p} A^{**}\rightarrow
	(A\otimes_{p} A)^{**}$ such that for $a,b\in A$ and $m\in
	A^{**}\otimes_{p} A^{**}$, the following holds;
	\begin{enumerate}
		\item [(i)] $\psi(a\otimes b)=a\otimes b $,
		\item [(ii)] $\psi(m)\cdot a=\psi(m\cdot a)$,\qquad
		$a\cdot\psi(m)=\psi(a\cdot m),$
		\item [(iii)] $\pi_{A}^{**}(\psi(m))=\pi_{A^{**}}(m),$
	\end{enumerate}
	see \cite[Lemma 1.7]{gha}. Set $\eta=\psi\circ\rho|_{A}:A\rightarrow (A\otimes_{p}A)^{**}.$ Clearly $\eta $ is a bounded linear map which satisfies 
	$$\eta(ab)=\psi\circ\rho|_{A}(ab)=\psi(a\cdot\rho|_{A}(b))=a\cdot \psi\circ\rho|_{A}(b)$$ 
	and 
	$$\phi(b)\eta(a)=\phi(b)\psi\circ\rho|_{A}(a)=\psi(\phi(b)\rho|_{A})=\psi\circ\rho|_{A}(ab)=\eta(ab).$$
	Also we have $$\tilde{\phi}\circ\pi^{**}_{A}\circ\eta(a)=\tilde{\phi}\circ\pi^{**}_{A}\circ\psi\circ\rho|_{A}(a)=\tilde{\phi}\circ\pi_{A^{**}}\circ\rho|_{A}(a)=\phi(a),$$
	for each $a\in A.$ It follows that $A$ is left $\phi$-biflat.
\end{proof}
\begin{Theorem}
	Let $A$ and $B$ be banach algebra with $\phi\in\Delta(A)$ and $\psi\in\Delta(B).$  If $A$ left $\phi$-biprojective and $B$ is $\psi$-biprojective, then $A\otimes_{p}B$ is left $\phi\otimes\psi$-biprojective.
\end{Theorem}
\begin{proof}
	Since $A$ left $\phi$-biprojective and $B$ is $\psi$-biprojective, there exist  bounded linear map $\rho_{A}:A\rightarrow A\otimes_{p}A$ and $\rho_{B}:B\rightarrow B\otimes_{p}B$ such that $$\rho_{A}(a_{1}a_{2})=a_{1}\cdot \rho_{A}(a_{2})=\phi(a_{2})\rho_{A}(a_{1}),\quad \phi\circ\pi_{A}\circ \rho_{A}=\phi,\quad (a_{1},a_{2}\in A)$$
	and $$\rho_{B}(b_{1}b_{2})=b_{1}\cdot \rho_{B}(b_{2})=\phi(b_{2})\rho_{B}(b_{1}),\quad \psi\circ\pi_{B}\circ \rho_{B}=\psi,\quad (b_{1},b_{2}\in B).$$ Let $\theta$ be an isometrical isomorphism from $(A\otimes_{p}A)\otimes_{p}(B\otimes_{p} B)$ into $(A\otimes_{p}B)\otimes_{p}(A\otimes_{p} B)$ which is given by $\theta (a_{1}\otimes a_{2}\otimes b_{1}\otimes b_{2})=a_{1}\otimes b_{1}\otimes a_{2}\otimes b_{2}$ for each $a_{1},a_2\in A$ and $b_1,b_2\in B.$
	Define $\rho=\theta\circ(\rho_{A}\otimes \rho_{B})$. So 
	 \begin{equation*}
	\begin{split}
	\rho((a_{1}\otimes b_{1})(a_{2}\otimes b_{2}))=\theta\circ(\rho_{A}\otimes \rho_{B})((a_{1}\otimes b_{1})(a_{2}\otimes b_{2}))&=\theta(\rho_{A}(a_{1}a_{2})\otimes \rho_{B}(b_{1}b_{2})\\
	&=\theta(a_1\cdot \rho_{A}(a_{2})\otimes b_1\cdot \rho_{B}(b_{2}))\\
	&=\theta((a_{1}\otimes b_{1})\cdot(\rho_{A}(a_{2})\otimes \rho_{B}(b_{2}))\\
	&=(a_{1}\otimes b_{1})\cdot \theta\circ(\rho_{A}\otimes \rho_{B})(a_2\otimes b_2),
	\end{split}
	\end{equation*}
	 for each $a_{1},a_2\in A$ and $b_1,b_2\in B.$ It follows that $\rho(xy)=x\cdot \rho(y)$ for each $x,y\in A\otimes_{p}B$. Also we have 
	 \begin{equation*}
	 \begin{split}
	 \phi\otimes\psi(a_{1}\otimes b_{1})\rho(a_{2}\otimes b_{2})=\phi(a_{1})\psi(b_{1})\theta\circ(\rho_{A}(a_2)\otimes \rho_B(b_2))&=\theta\circ(\phi(a_1)\rho_{A}(a_2)\otimes\psi(b_1) \rho_B(b_2))\\
	 &=\theta\circ(\rho_{A}(a_2a_1)\otimes \rho_B(b_2b_{1}))\\
	 &=\rho((a_2\otimes b_{2})(a_1\otimes b_1),
	 \end{split}
	 \end{equation*}
	 for each $a_{1},a_2\in A$ and $b_1,b_2\in B.$ So for each $x, y\in A\otimes_{p} B$ we have $\phi\otimes \psi(x)\rho(y)=\rho(yx)$. Note that $$\pi_{A\otimes_{p}B}\circ\theta(a_{1}\otimes a_{2}\otimes b_{1}\otimes b_{2})=\pi_{A\otimes_{p}B}(a_{1}\otimes b_{1}\otimes a_{2}\otimes b_{2})=\pi_{A}(a_{1}\otimes a_{2})\pi_{B}(b_1\otimes b_{2}),$$
	 it implies that $\pi_{A\otimes_{p}B}\circ\theta=\pi_{A}\otimes \pi_{B}$. Then 
	 \begin{equation*}
	 \begin{split}
(	\phi\otimes \psi)\circ \pi_{A\otimes_{p}B}\circ \rho(a\otimes b)&=(	\phi\otimes \psi)\circ \pi_{A\otimes_{p}B}\circ\theta \circ(\rho_{A}\otimes \rho_{B})(a\otimes b)\\
&=(\phi\otimes\psi)\circ(\pi_{A}\otimes \pi_{B})\circ(\rho_{A}\otimes \rho_{B})(a\otimes b)\\
&=\phi\circ\pi_{A}\circ\rho_{A}(a) \psi\circ\pi_{B}\circ\rho_{B}(b)\\
&=\phi(a)\psi(b)=\phi\otimes\psi(a\otimes b),
	 \end{split}
	 \end{equation*}
	 for each $a\in A$ and $b\in B.$ Therefore $(	\phi\otimes \psi)\circ \pi_{A\otimes_{p}B}\circ \rho(x)=	\phi\otimes \psi(x)$ for every $x\in A\otimes_p B.$ It follows that $A\otimes_p B$ is left $\phi\otimes\psi$-biprojective.
\end{proof}
Let $\hat{G}$ be the dual group of $G$
which consists of all non-zero continuous homomorphism
$\rho:G\rightarrow \mathbb{T}$.  It is well-known
that every character (multiplicative linear functional) $\phi\in\Delta(L^{1}(G))$ has the form
$\phi_{\rho}(f)=\int_{G}\overline{\rho(x)}f(x)dx$, where $dx$ is the
normalized Haar measure and $\rho\in\hat{G}$, for more details see
\cite[Theorem 23.7]{hew}.
Note that, since $L^1(G) $ is a closed ideal of the mearsure algebra $M(G)$, each character on $L^1(G)$ can be extended to $M(G).$ Note that for  a locally compact group $G$, we denote $A(G)$ for the Fourier algebra. The character space $\Delta(A(G))$ consists of all point evaluations $\phi_{x}$ for each $x\in G,$ where $$\phi_{x}(f)=f(x), \quad (f\in A(G)),$$
see\cite[Example 2.6]{kan}.
\begin{Theorem}
	Let $G$ be a locally compact group. Then $M(G)\otimes_{p} A(G)$ is left $\phi\otimes \psi$-biprojective if and only if $G$ is finite, where $\phi\in\Delta(L^{1}(G))$ and $\psi\in\Delta(A(G))$.
\end{Theorem}
\begin{proof}
Let $M(G)\otimes_{p} A(G)$ be left $\phi\otimes \psi$-biprojective. Let $e$ be the unit of $M(G)$ and $a_{0}$ be the element of $A(G)$ such that $\psi(a_{0})=1.$	Put $x_{0}=e\otimes a_{0}$. Clearly
$xx_{0}=x_{0}x$ and $\phi\otimes \psi(x_{0})=1,$ for every $x\in M(G)\otimes_{p} A(G).$ Now applying \cite[Lemma 2.2]{sahami left biprojective}   $M(G)\otimes_{p} A(G)$ is  left $\phi\otimes \psi$-contractible. Now using \cite[Theorem 3.14]{Nas} $M(G)$ is left $\phi$-contractible, so by \cite[Theorem 6.2]{Nas} $G$ is compact. Also by \cite[Theorem 3.14]{Nas} $A(G)$ is left $\psi$-contractible.  Thus by \cite[Proposition 6.6]{Nas} $G$ is discrete. Therefore $G$ is finite.

Converse is clear. 
\end{proof}
\begin{Theorem}
	Let $G$ be a locally compact group. Then $M(G)\otimes_{p} L^1(G)$ is left $\phi\otimes \psi$-biprojective if and only if $G$ is compact, where $\phi,\psi\in\Delta(L^{1}(G))$.
\end{Theorem}
\begin{proof}
	Suppose that $M(G)\otimes_{p} L^1(G)$ is left $\phi\otimes \psi$-biprojective. Let $e$ be the unit of $M(G)$ and $e_{\alpha}$ be a bounded approximate identity of $L^1(G).$ Clearly $e\otimes e_{\alpha}$ is a bounded approximate identity. Thus by \cite[Lemma 2.2]{sahami left biprojective} $M(G)\otimes_{p} L^1(G)$ is left $\phi\otimes \psi$-contractible. So \cite[Theorem 3.14]{Nas} $L^{1}(G)$ is left $\psi$-contractible. Then by \cite[Theorem 6.2]{Nas} $G$ is compact.
	
	For converse, suppose that  $G$ is compact. Then by \cite[Theorem 3.14]{Nas} M(G) is left $\phi$-contractible and by \cite[Theorem 3.14]{Nas} $L^{1}(G)$ is left $\psi$-contractible. Applying \cite[Theorem 3.14]{Nas} $M(G)\otimes_{p} L^1(G)$ is left $\phi\otimes \psi$-contractible. So by \cite[Lemma 2.1]{sahami left biprojective} $M(G)\otimes_{p} L^1(G)$ is left $\phi\otimes \psi$-biprojective.
\end{proof}
A Banach algebra $A$ is called left character biprojective (left character biflat) if $A$ is left $\phi$-biprojective (if $A$ is left $\phi$-biflat) for each $\phi\in\Delta(A),$ respectively.
\begin{Theorem}
	Let $G$ be a locally compact group. Then
$M(G)\otimes_{p} L^1(G)$ is  left character biprojective if and only if $G$ is finite.
\end{Theorem}
\begin{proof}
Let $M(G)\otimes_{p} L^1(G)$ be left character biprojective. So $M(G)\otimes_{p} L^1(G)$ is left $\phi\otimes\psi$-biprojective for each $\phi\in\Delta(M(G))$ and $\psi\in\Delta(L^{1}(G))$. So by similar arguments as in previous Proposition,  $M(G)$ left $\phi$-contractible for each $\phi\in\Delta(M(G))$. ُ
Since $M(G)$ is unital, by \cite[Corollary 6.2]{Nas} $G$ is finite.

Converse is clear.
\end{proof}
\begin{Theorem}
	Let $G$ be a locally compact group. Then
	$M(G)\otimes_{p} L^1(G)$ is  left character biflat  if and only if $G$ is a discrete amenable group.
\end{Theorem}
\begin{proof}
	Since $M(G)$ is unital and $L^{1}(G)$ has a bounded approximate identity, $M(G)\otimes_{p} L^1(G)$ has a bounded approximate identity. Thus by \cite[Lemma 2.1]{sahami left biflat} $M(G)\otimes_{p} L^1(G)$ is  left $\phi\otimes\psi$-amenable for each $\phi\in\Delta(M(G))$ and $\psi\in\Delta(L^1(G))$. So by \cite[Theorem 3.3]{kan} $M(G)$ is left $\phi$-amenable for each $\phi\in\Delta(M(G))$. Since $M(G)$ is unital, $M(G)$ character amenable. Therefore by  the main result of \cite{Sang}, $G$ is discrete and amenable.
	
	For converse, let $G$ be discrete and amenable. Then $M(G)\otimes_{p} L^1(G)=\ell^{1}(G)\otimes_p\ell^1(G)\cong \ell^1(G\times G)$. Applying Johnson's theorem (see \cite[Theorem 2.1.18]{run}) that $\ell^1(G\times G)$ is an amenable Banach algebra. So by \cite[Exercise 4.3.15]{run} $\ell^1(G\times G)$ biflat. Then  $\ell^1(G\times G)$ is left  character biflat.
\end{proof}
\begin{Proposition}
Let $G$ be an amenable group. Then
$A(G)\otimes_{p} L^1(G)$ is  left $\phi\otimes\psi$-biprojective  if and only if $G$ is finite.
\end{Proposition}
\begin{proof}
	Since $G$ is amenable, Leptin's Theorem \cite[Theorem 7.1.3]{run} gives that $A(G)$ has a bounded approximate identity. It is well-known that $L^1(G)$ has a bounded approximate identity. Then by \cite[Proposition 2.4]{sahami left biprojective}, left $\phi\otimes\psi$-biprojectivity of  $A(G)\otimes_{p} L^1(G)$ implies that $A(G)\otimes_{p} L^1(G)$ is left $\phi\otimes\psi$-contractible. So using \cite[Theorem 3.14]{Nas} gives that $A(G)$ is left $\phi$-contractible. Then by \cite[Proposition 6.6]{Nas} $G$ is discrete. Also by \cite[Theorem 3.14]{Nas} $L^{1}(G)$ is left $\psi$-contractible. Then \cite[Theorem 6.1]{Nas} implies that $G$ is compact. It follows that $G$ is finite.
	
	Converse is clear.
\end{proof}
\begin{Proposition}
	Let $G$ be a locally compact  group. Then
	$A(G)\oplus_{1} L^1(G)$ is  left  character biprojective  if and only if $G$ is finite.
\end{Proposition}
\begin{proof}
Suppose that $A(G)\oplus_{1} L^1(G)$ is left character biprojective. Let $\phi\in\Delta(A(G))$. Choose an element $a_{0}\in A(G)$ such that $\phi(a_{0})=1.$ Clearly the element $x_{0}=(a_{0},0)$ belongs to  $A(G)\oplus_{1} L^1(G)$ which $xx_{0}=x_{0}x$ and $\phi(x_{0})=1.$ Using \cite[Lemma 2.2]{sahami left biprojective}, left character biprojectivity of $A(G)\oplus_{1} L^1(G)$ implies that $A(G)\oplus_{1} L^1(G)$ is left $\phi$-contractible. Since $A(G)$ is a closed ideal in  	$A(G)\oplus_{1} L^1(G)$ and $\phi|_{A(G)}\neq 0$, by \cite[Proposition 3.8]{Nas}$A(G)$ is left $\phi$-contractible. So by \cite[Proposition 6.6]{Nas} $G$ is discrete. Thus $A(G)\oplus_{1} L^1(G)=A(G)\oplus_{1} \ell^1(G)$.
 We know that $\ell^1(G)$ has an identity $e$. Replacing $e$ with $a_{0}$ and $\psi$ with $\phi$ (for some $\psi\in\Delta(L^1(G))$) and following the same argument as above,  we can see that $\ell^1(G)$ is left $\psi$-contractible. Thus by \cite[Theorem 6.1]{Nas} $G$ is compact. Therefore $G$ must be finite.

 Converse is clear.
\end{proof}
A linear subspace $S(G)$ of $L^{1}(G)$ is said to be
a Segal algebra on $G$ if it satisfies the following conditions
\begin{enumerate}
	\item [(i)] $S(G)$ is  dense    in $L^{1}(G)$,
	\item [(ii)]  $S(G)$ with a norm $||\cdot||_{S(G)}$ is
	a Banach space and $|| f||_{L^{1}(G)}\leq|| f||_{S(G)}$ for every
	$f\in S(G)$,
	\item [(iii)] for $f\in S(G)$ and $y\in G$, we have $L_{y}(f)\in S(G)$ the map $y\mapsto L_{y} (f)$ from $G$ into $S(G)$ is continuous, where
	$L_{y}(f)(x)=f(y^{-1}x)$,
	\item [(iv)] $||L_{y}(f)||_{S(G)}=||f||_{S(G)}$ for every $f\in
	S(G)$ and $y\in G$.
\end{enumerate}
For various examples of Segal algebras, we refer the reader  to \cite{rei}. 

A locally compact group $G$ is called $SIN$, if it contains a foundamental family of compact invariant neibourhoods of the identity, see \cite[p. 86]{doran}.
\begin{Proposition}
	Let $G$ be a SIN  group. Then
	$S(G)\otimes_{p}S(G)$ is  left $\phi\otimes \psi$-biprojective  if and only if $G$ is compact, for some $\phi\in\Delta(S(G))$.
\end{Proposition}
\begin{proof}
Let $S(G)\otimes_{p}S(G)$ be left $\phi\otimes \phi$-biprojective. Since $G$ is a $SIN$ group, the main result of \cite{kot} gives that $S(G)$ has a central approximate identity. It follows that there exists an element $x_{0}\in S(G)$ such that $xx_{0}=x_{0}x$ and $\phi(x_{0})=1,$ for each $x\in S(G)$. Set $u_{0}=x_{0}\otimes x_{0}$. It is easy to see that $uu_{0}=u_{0}u$ and $\phi\otimes \phi(u_{0})=1,$ for every $u\in S(G)\otimes_{p}S(G).$ Using \cite[Lemma 2.2]{sahami left biprojective} left $\phi\otimes \phi$-biprojectivity of $S(G)\otimes_{p}S(G)$ follows that $S(G)\otimes_{p}S(G)$ is left  $\phi\otimes \psi$-contractible. By \cite[Theorem 3.14]{Nas} $S(G) $ is left $\phi$-contractible. Thus \cite[Theorem 3.3]{Ala} gives that $G$ is compact.

For converse, suppose that $G$ is compact. Then by \cite[Theorem 3.3]{Ala} $S(G) $ is left $\phi$-contractible. So by \cite[Theorem 3.14]{Nas} $S(G)\otimes_{p}S(G)$ be left $\phi\otimes \phi$-contractible. Applying \cite[Lemma 2.1]{sahami left biprojective} finishes the proof.
\end{proof}
\begin{small}
	
\end{small}
\end{document}